\documentclass[12 pt, onecolumn, draftcls]{IEEEtran}  % Comment this line out
\IEEEoverridecommandlockouts                              % This command is only
\usepackage{epsfig} % for postscript graphics files
\usepackage{amsmath} % assumes amsmath package installed
\usepackage{amssymb}  % assumes amsmath package installed
\usepackage{color}
\usepackage{url}

\def\0{{\bf 0}}
\def\1{{\bf 1}}
\def\d{{\bf d}}
\def\G{{\cal G}}
\def\A{{\cal A}}
\def\p{{\bf p}}

\def\R{{\mathbf{R}}}
\def\e{{\bf e}}

\long\def\jnt#1{{#1}}
\long\def\blue#1{{#1}}
\long\def\comm#1{}

\def\i[#1]#2{\langle #1,#2 \rangle_d}
\def\old#1{}
\def\aot#1{{#1}}
\def\ao#1{{#1}}
\def\jntn#1{{#1}}

\def\red#1{{#1}}
\newtheorem{theorem}{{\indent Theorem}}

\newtheorem{proposition}[theorem]{\indent {Proposition}}
\newtheorem{corollary}[theorem]{\indent {Corollary}}
\newtheorem{lemma}[theorem]{\indent {Lemma}}
\newtheorem{remark}[theorem]{\indent {Remark}}

\title{\LARGE \bf Degree Fluctuations and the Convergence Time of Consensus Algorithms\thanks{Research partially supported by the NSF under grant CMMI-0856063.}}
 \author{Alex Olshevsky\thanks{Alex Olshevsky is with the Department of Industrial and Enterprise Systems Engineering, University of Illinois at Urbana-Champaign, Email: \texttt{aolshev2@illinois.edu}} and  John N. Tsitsiklis \thanks{John N. Tsitsiklis is with the Laboratory for Information and
Decision Systems, Department of Electrical Engineering and Computer
Science, Massachusetts Institute of Technology. Email:
\texttt{jnt}\texttt{@mit.edu}.  }\thanks{A preliminary version of this paper appeared in the Proceedings of the 50th IEEE Conference on Decision and Control. The current
manuscript contains expanded proofs and new results relative to the conference version. } }
\begin{document}
%\date{}
\maketitle

\begin{abstract}
\noindent
We consider a consensus algorithm in which every node in a sequence of undirected, $B$-connected graphs assigns equal  weight to each  of its neighbors. Under the assumption that the degree of each node is fixed (except \jntn{for} times when \jntn{the node} has no connections to other nodes), we show  that consensus is achieved within a given accuracy $\epsilon $ \ao{on $n$ nodes} in time $\aot{B+}4n^3 B \blue{\ln}(2n/\epsilon)$. Because there is a direct  relation between consensus algorithms in time-varying environments and  \ao{in}homogeneous random walks, our result also translates into a general  statement on such random walks. \ao{Moreover,} we give a simple proof of a result of Cao, Spielman, and Morse that the \red{worst case} convergence time becomes \ao{exponentially large in the number of nodes $n$} under slight relaxation of  the  degree constancy assumption.
\end{abstract}

\section{Introduction}
\jnt{Consensus algorithms are a class of iterative update schemes that are commonly
used as building blocks for the design of distributed \ao{control laws. Their main advantage is robustness} in the presence of
time-varying environments and unexpected communication link failures.
Consensus algorithms have attracted significant interest in a variety of contexts such as distributed optimization \cite{TBA86}, \cite{NO09} \ao{coverage control \cite{GCB07}}, and many other contexts involving networks in which central control is absent and communication capabilities \blue{are} time-varying.}

\jnt{While the convergence properties of consensus algorithms in time-varying environments are well understood, \blue{much less is known about the} corresponding convergence times. An inspection of the classical convergence proofs \blue{(\cite{BT, JLM03})} leads to convergence time upper bounds that grow exponentially with the number of nodes. It is then natural to \blue{look for} conditions under which the convergence time only grows polynomially, and this is the subject of this \blue{paper.}}

\jnt{In our main result, we show that a consensus algorithm in which every node assigns equal weight to each of its neighbors \ao{in a sequence of
undirected graphs}  has polynomial convergence time if the degree of any given node is constant in time (except possibly during the times when \jntn{the node} has no connections to other nodes). 
%Because there is a direct relation between consensus algorithms in time-varying environments and inhomogeneous random walks (, our result also translates into a general statement on such random walks.
}

\subsection{\jnt{Model, notation, and background}}  
\jnt{In this subsection, we define our notation, the model of interest, and some background on consensus algorithms.}
 
We will consider only undirected graphs in this paper; this will often be stated explicitly, but when unstated
every graph should be understood to be undirected by
default. Given \jnt{a} graph $G$, we will use $N_{i}(G)$ to
denote the set of neighbors of node $i$. Given a sequence of graphs $G(0), G(1), \ldots, G(k-1)$, we will use \jnt{the simpler notation $N_i(t),
d_i(t)$ \jnt{in place of} $N_i(G(t))$, $d_i(G(t))$, and we will \jnt{make a similar simplification for other variables of interest.} 

\jnt{We are interested in analyzing a consensus algorithm in which a node assigns equal weight to each one of its neighbors. We consider $n$ nodes and 
assume that at each discrete time $t$, node $i$ stores a real number $x_i(t)$. We let $x(t)=(x_1(t),\ldots,x_n(t))^T$.}
For any given} sequence of graphs $G(0), G(1), G(2), \ldots$, all on the node set $\{1, \ldots, n\}$, and \jnt{any} initial vector $x(0)$, the algorithm is \jnt{described by the update equation}
\begin{equation} \label{basicupdate} x_i(t+1) = \frac{1}{d_{i}(t)} \sum_{j \in N_{i}(t)}
x_j(t), \qquad \jnt{i=1,\ldots,n,}\end{equation}  
\jnt{which can also be written in the form
\begin{equation} \label{matrixupdate} x(t+1)=A(t)x(t),\end{equation} for a suitably defined sequence of matrices $A(0),$  $A(1),\ldots,$ $A(t-1)$.} 
The graphs $G(t)$, which appear in the above update rule through $d_i(t)$ and $N_i(t)$, 
correspond to information flow among the agents;
the edge \blue{$(i,j)$} is \jnt{present} in $G(t)$ if and only if agent \jnt{$i$} uses the value $x_{\jnt{j}}(t)$ of agent $\jnt{j}$ in its
update at time $t$. \jnt{To reflect the fact that} every agent always has access to its own information, we assume
that every graph $G(t)$ contains all the self-loops $(i,i)$; as a consequence, $d_i(t) \geq 1$ for all $i,t$. \blue{Note that we have $[A(t)]_{ij}>0$ if and only if $(i,j)$ is an edge in $G(t)$.}

We will say that the graph sequence $G(t)$ is $B$-connected if, for every $k \geq 0$, the graph obtained by taking the union of the	 edge sets of $G(kB), G(kB+1), \ldots, G((k+1)B-1)$
is connected. \jnt{It is well known} (\cite{TBA86, JLM03}) that if the graph sequence is $B$-connected for some positive integer $B$, then
\jnt{every component of $x(t)$ converges to a common value. In this paper, we focus on} the
convergence rate of this process in some natural settings. \jnt{To quantify the progress of the algorithm towards consensus,}
we will use the function $S(\jnt{x}) = \max_i x_i - \min_i x_i.$ \aot{For any $\epsilon>0$, 
a sequence of  stochastic matrices $A(0), A(1), \ldots, A(k-1)$
{\em
results in $\epsilon$-consensus} if \[ S(\jnt{A(k-1) \cdots A(1) A(0) x(k)}) \leq \epsilon S(\jnt{x(0)}) \] for all initial vectors $x(0)$; alternatively, a sequence
of graphs $G(0), G(1), \ldots$ achieves $\epsilon$-consensus if the sequence of matrices $A(t)$ defined by Equations (\ref{basicupdate}) and  (\ref{matrixupdate})
achieves $\epsilon$-consensus.} 

As mentioned previously, we will \jnt{focus} on graph sequences in which every graph $G(t)$ is undirected. There are a
number of reasons to be especially interested in undirected graphs within the context of consensus.  For example, $G(t)$
is undirected if: (i) $G(t)$  contains all the edges between agents that are physically within some distance of
each other; (ii) $G(t)$ contains all the edges between agents \jnt{that} have line-of-sight views of each other;
(iii) $G(t)$ contains the edges corresponding to pairs of agents  \jnt{that} can send messages to each other
using a protocol \jnt{that} relies on \jntn{acknowledgments.}

It is \jnt{an immediate} consequence of \jnt{existing convergence proofs} (\blue{\cite{BT}}, \cite{JLM03}) that 
\jnt{any
sequence of $C n^{nB} \blue{\ln} (1/\epsilon)$ undirected $B$-connected graphs, with self-loops at every node, results in
$\epsilon$-consensus. Here,} $C$ is a constant
\jnt{that} does not depend on the problem parameters $n$, $B$\jntn{,} and $\epsilon$. 
\jnt{We are interested in simple conditions on the graph sequence under which the undesirable $O(n^{nB})$ scaling becomes polynomial in $n$
and $B$.}

\subsection{Our results} Our contributions are as follows. First, in Section \ref{mainproof}, we prove our main result.

\begin{theorem} \jnt{Consider a sequence $G(0), G(1), \ldots, G(k-1)$ of 
$B$-connected undirected graphs with self-loops at
each node. Suppose that for each $i$ there exists some $d_i$ such that $d_i(t) \in \{1,d_i\}$ for all $t$ (note that $d_i(t)=1$ means node $i$
has no links to any other node).
 If the length $k$ of the graph sequence is at least $\aot{B+}4 n^3 B \blue{\ln} \frac{2 { n}}{\epsilon}$,
then $\epsilon$-consensus is achieved.} \label{mainthm} \end{theorem}

%\bigskip
%
%\jnt{To put Theorem \ref{mainthm} in perspective, we note that \ao{polynomial convergence times} to consensus
%were only known for the cases where:
%\begin{itemize}
%\item
%[(a)] the graphs $G(t)$ are the same at each time $t$, undirected, connected, with all self-loops present \cite{LO81}; or
%\item[(b)]
%in addition to some natural connectivity assumptions, the underlying iteration matrices are doubly stochastic, which, for the consensus algorithms considered here amounts to an assumption that each graph $G(t)$ is regular \cite{NOOT09}.
%\end{itemize}

\old{As will be seen in the proof, it reveals that the essential property that results in a polynomial time to achieve $\epsilon$-consensus is the following: there exists a particular linear function of the vectors $x(t)$
that remains invariant, despite the time-varying nature of the underlying graphs. (For the case of doubly stochastic iteration matrices, this linear function is the sum of the components of $x(t)$. Our result suggests that there is nothing special about this particular linear combination.)}

%Up to
%differences in the constants, our result may
%be viewed as a generalization of the result from \cite{LO81}, since the constant graph
%naturally has degrees that
%do not change. Despite proving a considerably more general result, our (self-contained)
%proof is about half as long as the earlier
%proof in \cite{LO81}.

In Section \ref{markov}, we give an interpretation of \jnt{our} results in
terms of Markov chains. Theorem \ref{mainthm} can be interpreted
as providing a sufficient condition for a random walk on a time-varying graph
to forget its initial distribution in polynomial time. 

\jnt{In} Section \ref{counterexamples}, we \red{capitalize on the Markov chain interpretation and} provide a simple proof that relaxing
the assumptions of \jnt{Theorem \ref{mainthm}} even slightly \jnt{can} lead to a convergence time which is
exponential in $n$. Specifically, if we replace the assumption that each $d_i(t)$ is independent of $t$  with the
weaker assumption that the \jnt{sorted} degree sequence (\jnt{say,} in non-increasing order) is independent of $t$ (thus allowing nodes to ``swap'' degrees), exponential convergence time is possible. \jnt{This was proved earlier by Cao, Spielman, and Morse (although unpublished) \cite{cao} and our contribution is to provide a simple proof}.

%\begin{theorem} \label{thm:exp} Suppose the degree sequence of the graphs stays
%constant, i.e. for all $t,t'$ \[ d(G(t)) = d(G(t')).\] Then \[
%T(n,\epsilon) \geq \left( \frac{n}{2} \right)^{n/2} \log
%\frac{1}{\epsilon}.\]
%\end{theorem}

In summary: for \jnt{undirected} $B$-connected graphs with self-loops,  unchanging degrees is a sufficient condition for
polynomial time convergence, but relaxing it even slightly by allowing the nodes to ``swap'' degrees leads to the
possibility of exponential
convergence time.

%We prove Theorem \ref{thm:poly} through an analysis of common
%quadratic Lyapunov functions for this system. Theorem \ref{thm:exp}
%is proved by an improved analysis of an example of Cao, Spielman,
%and Morse \cite{cao}, who used it to establish the lower bound
%$T(n,\epsilon)\geq 2^{n/2} \log (1/\epsilon)$ in that setting.

%We offer some conclusions in
%Section \ref{concl}.
\subsection{Previous work} There is considerable \jnt{and growing} %amount of 
literature \jnt{on} the convergence time of
consensus algorithms. The recent paper \cite{JLM03} \jntn{amplified the interest}  in consensus algorithms and spawned a vast \jntn{subsequent} literature, 
which is impossible to survey here. 
We only mention papers that are closest to our own work, omitting references to the literature \jnt{on} various aspects of consensus convergence times \jnt{that} we do not address here.

Worst-case \jnt{upper} bounds on the \jnt{convergence times} of consensus algorithms \jnt{have been} established in \cite{CSM05, CMA08a, CMA08b, AB1, AB2, Chazelle09}. The papers \cite{CSM05, CMA08a, CMA08b} considered a setting slightly more general than ours, and established exponential upper bounds. The papers \cite{AB1, AB2} \jnt{addressed} the convergence times of consensus algorithms in terms of spanning trees \jnt{that} capture the information flow
between the nodes. It was observed that in several cases this approach produces tight estimates of the convergence times.  We mention also \cite{MBCF07} which derives a polynomial-time upper \jntn{bound} on the time and total communication complexity required by a network of robotic agents to implement various deployment and coordination schemes.  \red{Reference} \cite{Chazelle09} take\red{s} a geometric approach, and consider\red{s} the convergence time \jnt{in a somewhat different model, involving   interactions between geographic\jntn{ally} nearest neighbors.} \red{It finds} that the convergence time is quite high (either singly exponential or iterated exponential, depending on the 
\blue{model}). Random walks on undirected graphs such as  \jntn{considered} here are special cases of reversible agreement systems considered in the related work \cite{Chazelle11} (see also \cite{Chazelle10b} and \cite{Chazelle10a}).  Our proo\jntn{f} techniques are heavily influenced by the classic paper \cite{LO81}
and share some similarities with those used in the recent work \cite{SZ07}, which used similar
ideas to bound the convergence time of some inhomogenuous Markov chains.  There are also similarities with the recent work \cite{AKL08} on the cover time
of time-varying graphs.

% The original papers \cite{TBA86, JLM03} also considered the effect of delays on convergence; some more recent work on this subject may
%be found in \cite{CMA08b} and \cite{BNO}. 

Our work differs from these papers in that it studies time-varying, $B$-connected graphs and \jntn{establishes} convergence time bounds \jntn{that}
are polynomial in $n$ and $B$. To the best of our
knowledge, polynomial bounds on the \jnt{particular consensus algorithm considered in this paper} had previously been derived earlier only in %the papers 
\cite{LO81} (under the assumption that the graph is fixed, undirected, with self-loops at every node), \cite{NOOT09} (in the case when the matrix is doubly stochastic, which in our \jntn{setting} corresponds to a sequence of regular graphs $G(t)$). For the special case of graphs that are connected at every time step \jntn{($B=1$)}, the result
has been apparently discovered independently by Chazelle \cite{Chazelle11b} and the 
authors \cite{Otalk}. \jntn{Our added generality allows for both disconnected graphs in which the degrees are kept constant, as well as the case where nodes temporarily disconnect from the network, setting their degree to one.}

\section{Proof of Theorem \ref{mainthm}\label{mainproof} }

\jnt{As in the statement} 
of Theorem \ref{mainthm}, \jnt{we assume} that we are given a sequence of \jnt{undirected $B$-connected} 
graphs $G(\blue{0}), G(1), \ldots$, with self-loops at each node, such
that $d_i(t)$ equals either $d_i$ or $1$. Observe that $d_i>1$ for all $i=1, \ldots, n$, since else the sequence of graphs $G(t)$ could not be $B$-connected. We will use the notation ${\cal G}$ to refer to the class of \jnt{undirected} graphs with self-loops at every node such that the degree of node $i$ either $1$ or $d_i$. Note that the definition of ${\cal G}$ depends on the values $d_1, \ldots, d_n$. %We let $D$ be the $n \times n$ diagonal matrix whose $i$th diagonal entry is $d_i$.  
%We will use $E'(G)$ to denote the edge
%set of a graph $G$. We will sometimes find it convenient to use the notation $E(G)$ to refer to the set \jnt{of {\em unordered pairs} $(i,j)$ such that the {\em %ordered pairs} $(i,j)$ and $(j,i)$ belong to} $E'(G)$.

Given an undirected graph $G$, we define the update matrix $A(G)$ by \begin{equation*}
[A(G)]_{ij} =
\begin{cases} {1}/{d_i(G)}, & \text{ if } j \in N_i(G),
\\
0, &\text{ \jnt{otherwise}.}
\end{cases}
\end{equation*} We use $A(t)$ as \jnt{a} shorthand for $A(G(t))$, so that Eq.\  (\ref{basicupdate}) can be written as 
\begin{equation}x(t+1) = A(t) x(t).
\label{eq:ax}\end{equation} 
Conversely, given an update matrix $A$ of the above form,  we will use $G(A)$ to denote the graph $G$ whose update
matrix is $A$. We use $\A$  to denote
the set of update matrices $A(G)$ \jnt{associated with} graphs $G \in \G$. We define 
$\d$ to be the vector $\d = [d_1, d_2, \ldots, d_n]^T$; a simple calculation shows that $\d^T A = \d^T$ for all $A \in \A$. Finally,
we use $D$ to denote the matrix whose \jntn{$i$th} diagonal element is $d_i$. 
%Next, we will need to be able to refer to the order of eigenvalues of matrices $A \in \A$, prompting the following definition.

%\begin{definition} For a matrix $B$ with real eigenvalues, we will use $\lambda_i(B)$ to refer to the $i$'th largest eigenvalue of $A$, e.g., $\lambda_1(B)$ refers to the largest eigenvalue of $B$, $\lambda_2(B)$ refers to the second largest and so on. Observe that each $A \in \A$ has real eigenvalues as a consequence of the previous lemma, with $\lambda_1(A)=1$ due to the stochasticity of $A$. \end{definition}

We begin by identifying a weighted \jnt{average that is} preserved by the iteration $x(t+1)=A(t) x(t)$.
For any vector $y$, we let \[ \bar{y} = \frac{\d^T y}{\d^T \1} = \frac{\sum_{i=1}^n d_i y_i}{\sum_{i=1}^n d_i}, \] 
\jntn{where $\1$ is the vector with entries equal to 1.}
Observe that for any $A \in \A$,
\[ \overline{Ay} = \frac{\d^T A y}{\d^T \1} = \frac{\d^T y}{\d^T 1} = \bar{y}. \] Consequently, if $x(t)$ evolves according to Eq. (\ref{eq:ax}), then
$\overline{x(t)} = \overline{x(0)}$, which we will from now on denote simply by $\bar{x}$. 
% This follows straightforwardly from the self-adjointness and stochasticity of any $A \in \A$. 
%\begin{lemma} F \end{lemma} \begin{proof} \[ \overline{Ax}  = \frac{\i[Ax]{\1}}{\i[\1]{\1}} = \frac{\i[x]{A\1}}{\i[\1]{\1}} = %\frac{\i[x]{\1}}{\i[\1]{\1}} = \bar{x},\] where we used the self-adjointness of $A$ proved in
%Lemma \ref{selfadjoint}, as well as the fact that $A$ is stochastic so $A \1 = \1$.
%\end{proof}

With these preliminaries in place, we now proceed to the main part of our analysis, which \jnt{is} based on the pair of Lyapunov functions\[ V(x) = x^T D x = \sum_{i=1}^n d_i x_i^2, ~~\mbox{ and } ~~ V'(x) = \sum_{i=1}^n d_i (x_i-\bar{x})^2.
\] We will adopt the more convenient notation $V(t)$ for $V(x(t))$ and similarly $V'(t)$ for $V'(x(t))$. 

Our first lemma provides a convenient identity for matrices in $\A$.

\begin{lemma} \label{decomp} For any $A \in \A$ such that $G(A)$ is connected \jntn{(and in particular, every node $i$ has degree $d_i$),}
\[  A^T D A = D - \sum_{i<j} w_{ij} (\e_i- \e_j) (\e_i- \e_j)^T, \] where
$w_{ij}$ is the \jntn{$(i,j)$-th} entry of $A^TDA$. 
\end{lemma}

\begin{remark} This was proven in \cite{TN11} and is a generalized version of a decomposition from \cite{XB04, NOOT09}. 
It may be quickly verified by checking that both sides of the equation are symmetric, have identical row
sums, and whenever $i<j$, the $(i,j)$\jntn{-}th element of both sides is $w_{ij}$. The equality of \jntn{the two} sides then immediately follows. 
\end{remark}

Our next lemma \jnt{quantifies the decrease of $V(\cdot)$ when a vector $x$ is multiplied} by some matrix $A \in \A$ \jntn{associated} with a connected graph $G(A)$.

\begin{lemma} \label{vdecrease} Fix $x \in \R^n$ and let $i: \{1, \ldots, n\} \rightarrow \{1, \ldots, n\}$ be a permutation such that $x_{i[1]} \leq x_{i[2]} \leq \cdots \leq x_{i[n]}$. For any $A \in \A$ such that $G(A)$ is connected,  
\[ V(Ax) \leq V(x) - \frac{1}{2} \sum_{l=1}^{n-1} (x_{i[l+1]}-x_{i[l]})^2. \]
\end{lemma} 

\begin{proof} We may suppose without loss of generality that $x_1 \leq x_2 \leq \cdots \leq x_n$. Using Lemma~\ref{decomp}, 
\[ V(Ax) = (Ax)^T D (Ax) = x^T A^T D A x = V(x) - \sum_{i<j} w_{ij} (x_i - x_j)^2. \]  From the definitions of $w_{ij}, A$, and $D$, we
have that
\[ w_{ij} = \sum_{k \in N(i) \cap N(j)} \frac{1}{d(k)}, \] and so
\begin{equation} \label{lossbound}  V(Ax) = V(x) -  \sum_{i < j} (x_i - x_j)^2 \sum_{k \in N(i) \cap N(j)} \frac{1}{d(k)}. \end{equation} 

Observe that if $l<k$, then
\[ (x_k - x_l)^2 \geq (x_{l+1} - x_l)^2 + (x_{l+2} - x_{l+1})^2 + \cdots + (x_k - x_{k-1})^2 \] Applying this to each term of Eq. (\ref{lossbound}), we have that
\[ V(Ax) \leq V(x) - \sum_{i=1}^{n-1} W_i (x_i - x_{i+1})^2, \] where \begin{equation} \label{wdef} W_i = \sum_{k \leq i, ~l \geq i+1} \sum_{~m \in N(k) \cap N(l)} \frac{1}{d(m)} \end{equation}

We finish the proof by arguing that $W_i \geq 1/2$ for all \jntn{$i\leq n-1$.} Indeed, by \jntn{the connectivity} of $G(A)$, there is some node $j$ in $\{1, \ldots, i\}$ such that $j$ is connected to a node in $\{i+1, \ldots, n\}$. Let $d^+$ be the number of neighbors of node $j$ in $\{i+1, \ldots,n\}$ and $d^-$ be the number of neighbors of node $j$ in $\{1, \ldots, i\}$; naturally, $d_j = d^+ + d^-$ and both $d^+, d^-$ are at least $1$: the former by \jntn{the} definition of $j$\jntn{,} and the latter because node $j$ has a self-loop. Observe that the contribution to $W_i$ in Eq. (\ref{wdef})\jntn{,} by running $k$ over all the $d^-$ neighbors of $j$ in $\{1,\ldots,i\}$ and running $l$ over all $d^+$ neighbors of $j$ in $\{i+1, \ldots, n\}$\jntn{,} is at least \[ d^+ d^- \frac{1}{d_j} \geq \frac{d_j-1}{d_j} \geq \frac{1}{2}, \] where the final inequality is justified \jntn{because the connectivity of $G(A)$ implies that $d_j \geq 2$.} This concludes the proof. 
\end{proof} 

\begin{remark} We note that $V(Ax) \leq V(x)$, even if $G(A)$ is not connected; this follows by applying Eq. (\ref{lossbound}) to each connected component of 
$G(A)$. \label{nonincrease}
\end{remark}
\begin{lemma} Suppose \jntn{that} $x(t)$ evolves according to Eq. (\ref{eq:ax})\jntn{,} where $G(A(t))$ is a sequence of $B$-connected graphs from $\G$. Let $i: \{1, \ldots, n\} \rightarrow \{1, \ldots, n\}$ be a permutation such that $x_{i[1]}(kB) \leq x_{i[2]}(kB) \leq \cdots \leq x_{i[n]}(kB)$. Then, 
\[ V(x((k+1)B)) \leq V(x(kB)) - \frac{1}{2} \sum_{\jntn{l}=1}^{\jntn{n}} (x_{i[l+1]}(kB)-x_{i[l]}(kB))^2. \] \label{bconndec}
\end{lemma}

\begin{proof} It suffices to prove this under the assumption that $x_1(kB) < x_2(kB) < \cdots < x_n(kB)$; the general case then follows by a \jntn{continuity} argument. We apply the bound of Lemma \ref{vdecrease} at each time $t=kB, \ldots, (k+1)B-1$ to \jntn{each} connected component of $G(t)$. This yields that
\begin{equation} \label{intermediatesumbound} V((k+1)B) \leq V(kB) - \frac{1}{2} \sum_{t=kB}^{(k+1)B-1} \sum_{(q,l) \in C(t)} (x_q(t) - x_l(t))^2 \end{equation} Here, $C(t)$ contains all the pairs $(q,l)$ such that 
there is some component of $G(t)$ containing both $q$ and $l$, and $x_{\jntn{q}}(t)$ immediately follows $x_{\jntn{l}}(t)$ when the nodes in that
component are ordered according to increasing values of $x$. 

%We prove this lemma by arguing that for each $i=1,\ldots,n-1$, there is a  term in the sum on the right hand side of Eq. (\ref{intermediatesumbound}) which is 
%at least as big as $(x_{i[l+1]}-x_{i[l]})^2$, and, moreover, these terms are different for different $i$.   

We then observe that for every $i=1, \ldots, n-1$ there is a first time $t$ between $kB$ and $(k+1)B-1$ when there is a link between a node in $\{1, \ldots, i\}$ and a node in $\{i+1, \ldots, n\}$. Note that because there have been no links between $\{1, \ldots, i\}$ and $\{i+1, \ldots, n\}$ from time $kB$ to time $t-1$,  we have that
\[ \max_{j = 1, \ldots, i} x_j(t) \leq x_i(kB) < x_{i+1}(kB) \leq \min_{j=i+1, \ldots, n} x_j (t). \] Moreover, at time $t$, the sum on the right-hand side of Eq. (\ref{intermediatesumbound}) will contain the term $(x_{i'}(t) - x_{i''}(t))^2$ where $i' \in {\rm arg}~ \max_{j=1, \ldots, i} x_j(t)$ and $i'' \in {\rm arg}~ \min_{j=i+1, \ldots, n} x_j(t)$.  We conclude that it is possible to associate with every $i$ some triplet $i',i'',t$ such that $t \in [kB, (k+1)B-1]$,  $(i',i'') \in C(t)$ and $(x_i(kB)-x_{i+1}(kB))^2 \leq (x_{i'}(t)-x_{i''}(t))^2$.

To complete the proof, we argue that distinct $i$ are associated with distinct triplets $i',i'',t$. Indeed, we associate $i$ with $i',i'',t$ only if  $x_{i'}(t) = \max_{j = 1, \ldots, i} x_j(t)$ and there have been no links between $\{1, \ldots, i\}$ and $\{i+1, \ldots, n\}$ from time $kB$ to time $t-1$. Consequently if two indices $i_1<i_2$ are associated with the same triplet, it follows that ${\rm arg} ~ \max_{j = 1, \ldots, i_1} x_j(t) \cap {\rm arg} ~ \max_{j=1, \ldots, i_2} x_j(t) \neq \emptyset$ which cannot be: at time $kB$, $x_{i_2}(kB) \geq x_{i_1+1}(kB) > \max_{j=1, \ldots, i_1} x_j(kB)$ and no link between a node in $\{1, \ldots, i_1\}$ and a node $\{i_1+1, \ldots, n\}$ occured from time $kB$ to time $t-1$.
\end{proof}

The following lemma may be verified through a direct calculation.

\begin{lemma} Suppose $u_1, \ldots, u_n$ and $w_1, \ldots, w_n$ are numbers satisfying
\[ \sum_{i=1}^n d_i u_i = \sum_{i=1}^n d_i w_i. \] Then 
\[ \sum_{i=1}^n d_i (u_i - z)^2 - \sum_{i=1}^n d_i (w_i - z)^2 \] is a 
constant independent of the number $z$. \label{zchange}
\end{lemma}

\begin{corollary} Suppose $x(t)$ evolves according to Eq. (\ref{eq:ax}) where $G(A(t))$ is a sequence of $B$-connected graphs from $\G$. Let $i: \{1, \ldots, n\} \rightarrow \{1, \ldots, n\}$ be a permutation such that $x_{i[1]}(kB) \leq x_{i[2]}(kB) \leq \cdots \leq x_{i[n]}(kB)$. Then, 
\[ V'(x((k+1)B)) \leq V'(x(kB)) - \frac{1}{2} \sum_{k=1}^n (x_{i[l+1]}(kB)-x_{i[l]}(kB))^2. \] \label{vprime}
\end{corollary} 

\begin{proof} Lemma \ref{bconndec} may be restated as 
%\begin{eqnarray*} \sum_{i=1}^n d_i (x_i(kB) - 0)^2 - \sum_{i=1}^n d_i (x_i((k+1)B)-0)^2 & = &  V(x(kB)) - V(x((k+1)B)) \\
%& \leq & \frac{1}{2} \sum_{k=1}^n (x_{i[l+1]}(kB)-x_{i[l]}(kB))^2
%\end{eqnarray*} 
\[ \sum_{i=1}^n d_i (x_i(kB) - 0)^2 - \sum_{i=1}^n d_i (x_i((k+1)B)-0)^2 \leq  \frac{1}{2} \sum_{k=1}^n (x_{i[l+1]}(kB)-x_{i[l]}(kB))^2
\]
But since $\d^T x((k+1)B) = \d^T x(kB)$, we can apply Lemma \ref{zchange} to obtain 
\[ \sum_{i=1}^n d_i (x_i(kB) - \bar{x})^2 - \sum_{i=1}^n d_i (x_i((k+1)B)-\bar{x})^2 \leq \frac{1}{2} \sum_{k=1}^n (x_{i[l+1]}(kB)-x_{i[l]}(kB))^2, \] which
is a restatement of the current corollary.
\end{proof}

\begin{remark} An additional consequence of Lemma \ref{zchange} is that $V'(Ax) \leq V'(x)$ for all $A \in \A$. Remark \ref{nonincrease} had established this
property for $V(\cdot)$ and Lemma \ref{zchange} implies now the same property holds for $V'(\cdot)$. 
\end{remark}

\begin{lemma} For any $x$,   \[ \frac{\sum_{l=1}^{n-1} (x_{i[l+1]}-x_{i[l]})^2 }{V'(x)} \geq \frac{1}{n^2 d_{\rm max}},\]
where $d_{\rm max}$ is the largest of the degrees $d_i$. \label{diff}
\end{lemma}

%With our previous lemmas at hand, we can give a give a quick proof of this.
\begin{proof} \jntn{We employ} a variation of an argument
\jnt{first used} in \cite{LO81}. We first argue that we can make \ao{three} assumptions without loss of generality: \aot{1) that the components
of $x$ are sorted in nondecreasing order, i.e., $x_1 \leq x_2 \leq \cdots \leq x_n$;}
2) $\sum_i d_i x_i = 0$, since both the numerator and denominator on the left-hand side are invariant under the
addition of a constant to each component of $x$, \jntn{and in particular, $V(x)=V'(x)$;} 3) $\jntn{V'(x)=}\sum_i d_i x_i^2=1$, since the expression \jntn{on} the left-hand side remains invariant
under multiplication of each component of $x$ by a nonzero constant. 

\jntn{Let} $l$ be such that $d_l x_l^2 = \max_i d_i x_i^2$. 
Without loss of generality,
we can assume that $x_l>0$; else, we replace $x$ by $-x$. The condition that $\sum_i d_i x_i^2=1$
implies that $x_l \geq 1/\sqrt{n d_{\rm max}}$ while the condition that $\sum_i d_i x_i = 0$ implies \aot{$x_1<0$}.
Consequently, $x_l - x_{\aot{1}} \geq 1/\sqrt{n d_{\rm max}}$.We can write this as  \aot{
\[ (x_{2} - x_{1}) + (x_{3} - x_{2}) + \cdots + (x_{l} - x_{l-1}) \geq \frac{1}{\sqrt{n d_{\rm max}}}.\] Applying \jnt{the} Cauchy-Schwarz \jnt{inequality}, we get
\[ (l-1) \sum_{i=1}^{l-1} (x_{i+1} - x_{i})^2 \geq \frac{1}{n d_{\rm max}}.\] 
\jnt{We then use} the fact that $l-1 \leq n$ to complete the proof. }
\end{proof}

\jnt{We can now complete the proof of Theorem \ref{mainthm}.}
%Lemma 4 describes the decrease in the variance $V(x(t))$} in terms of $\lambda_2(A^2(t))$, and Lemma 5 gives us a way to upper bound the latter quantity.
%Putting it all together, we have the following proof.

\begin{proof}[Proof of Theorem \ref{mainthm}] From Corollary \ref{vprime} and Lemma \ref{diff}, we have that for all integer $k \geq 0$,
\[ V'((k+1)B) \leq (1 - \frac{1}{2n^3}) V'(kB). \]

\jnt{Because the definition of $\epsilon$-consensus is in terms of $S(x)$ rather than $V'(x)$, we need to relate these two quantities.}
%Next, we relate $T'(n,\epsilon)$ to our original metric of convergence $T(n,\epsilon)$.
%, which we defined as the time
%it takes $S(t)=\max_i x_i(t) - \min_i x_i(t)$ to permanently shrink by a factor of $\epsilon$.
On the one hand, \jnt{for every $x$, we have} \[  V'(\jnt{x}) = \sum_{i=1}^n d_i(\jnt{x_i} - \bar{x})^2 \leq n \sum_{i=1}^n (\jnt{x_i} - \bar{x})^2 \leq n^2 S^2(\jnt{x}) \label{upperbound}. \]
On the other hand, 
\jnt{for every $x$, we have
$$V'(x) \geq \max_i (x_i-\bar{x})^2 \geq \frac{1}{4}
(\max_i x_i -\min_i x_i)^2 =\frac{1}{4} S^2(x).$$
Suppose that \aot{$t\geq B + 4 B n^3\ln(\blue{2} { n}/\epsilon)$.  Then at least $\lceil 4 n^3 \ln 2n/\epsilon \rceil$ time periods\footnote{The notation $\lceil x \rceil$ means the smallest integer which is at least $x$.} of length $B$ have 
passed, and therefore} 
\blue{\[
S(x(t)) \leq \sqrt{4 V'(x(t))}  \leq  2\Big(1 -\frac{1}{2n^3}\Big)^{{ 4 } n^3\ln(2 { n}/\epsilon) { (1/2)}} { \sqrt{V'(x(0))}}  
\leq  2 { n} e^{-\ln(2 { n}/\epsilon)}S(x(0)) =\epsilon S(x(0)).
\]}
(We have used here the inequality $(1-1/x)^x \leq e^{-1}$, for $x \geq 1$ as well as the fact that $V'(\cdot)$ is nonincreasing.)} 
%$G(t)$ is a connected graph, so there exists a path in $G(t)$ from the node with the largest $x_i(t)$ to the
%node with the smallest $x_i(t)$. This path can be chosen to have length at most $n$, and so there is at least one edge $(i,j)$
%on the path with a gap of $|x_i(t)-x_j(t)| \geq (1/n) S(t)$. This implies that $V(t) \geq \frac{1}{n^2} S(t)^2.$ Thus we have upper and lower bounds on $V(t)$:  \[ \frac{1}{n^2} S(t)^2 \leq V(t) \leq n^2 S(t)^2. \] The above equation implies that
%if we were to have $V(t) \leq \frac{\epsilon^2}{n^4} V(0),$ then we could conclude that $S(t)^2 \leq \epsilon^2 S(0),$ and
%consequently $S(t) \leq \epsilon S(0).$ Thus $T(n,\epsilon) \leq T'(n, \frac{\epsilon^2}{n^4})$. It only remains to bound this last
%quantity. We have that
%\[ T'(n, \frac{\epsilon^2}{n^4}) = \log_{(1-1/n^3)} \frac{\epsilon^2}{n^4} = \frac{\ln (\epsilon^2/n^4)}{\ln (1-1/n^3)}. \] Next we make use of
%the fact that $-\ln(1-x) \geq x/2$ for $x \geq 0$: 
%\[ T'(n, \frac{\epsilon^2}{n^4}) \leq 2 n^3 \ln \frac{n^4}{\epsilon^2} \leq 8 n^3 \ln (n/\epsilon).\]
\end{proof}

\section{Markov chain interpretation \label{markov}} 
\comm{In this section, the use of $t$, $k$ has been reversed. Fix in the journal version.} 
In this \jnt{section,} we give an alternative interpretation of the convergence time of a consensus algorithm in terms of \jnt{inhomogeneous} Markov
chains; this interpretation will be used in the next section. We refer the reader to the recent monograph \cite{LPW08} for the requisite background on Markov chains and random walks.

%We suppose that we are given a arbitrary sequence of (directed) graphs $G(1), G(2), \ldots$, which we will denote
%by $\Gb$.
We consider an inhomogeneous Markov chain whose transition probability matrix at time $k$ is $A(k)$. We fix some time $t$ and define
$$P=A(0) A(1)\cdots A(t-1).$$
This is the associated $t$-step transition probability matrix: the $(i,j)$-th entry of $P$, denoted by $p_{ij}$, is the probability that the state at time $t$ is $j$, given that the initial state is $i$. Let $\p_i$ be the vector whose $k$th component is $p_{ik}$; thus $\p_i^T$ is the $i$th row of $P$.

We address a question which is generic in the study of Markov chains, namely, whether the chain eventually 
``forgets'' its initial state, i.e., whether \blue{for all $i,j$, 
$\p_i - \p_j$ converges to zero as $t$ increases,}  and if so, at what rate.  
We will say that the sequence of matrices $A(0), A(1), \ldots, A(t-1)$ is {\em  $\epsilon$-forgetful} if for all $i,j$, we have 
\jnt{
$$\frac{1}{2} \sum_k |p_{ik}- p_{jk}| \leq \epsilon.$$
The} \blue{above} \jnt{quantity}, $\jnt{\frac{1}{2}} \max_{i,j} \|\p_i - \p_j\|_1$ is known as the {\em coefficient of ergodicity}
of the matrix $P$, and appears often in the study of consensus algorithms (see, for example, \cite{CSM05}). 
%The matrix $P$ can also be interpreted in terms of consensus updates: an initial vector $x$ is multiplied by the matrices $A(t-1),A(t-2),\ldots,A(0)$, to produce the %vector $Px$; note that the different matrices are now applied in the reverse order.  
\jnt{The result that follows relates
the times to achieve $\epsilon$-consensus or $\epsilon$-forgetfulness,}
\blue{and is essentially the same as Proposition 4.5 of \cite{LPW08}.}

\begin{proposition} The sequence of matrices $A(0), A(1), \ldots,$ $A(t-1)$ is $\epsilon$-forgetful if and only if the sequence of matrices $A(t-1),A(t-2),\ldots,A(0)$ results in $\epsilon$-consensus (i.e., $S(Px)\leq \epsilon S(x)$, for every vector $x$.) \label{forget} \end{proposition}

\begin{proof} Suppose that the matrix sequence $A(0), A(1), \ldots,$ $A(t-1)$ is $\epsilon$-forgetful, i.e., that $\frac{1}{2}\sum_k |p_{ik}-p_{jk}|\leq \epsilon$, for all $i$ and $j$.
Given a vector $x$, let $c=(\max_k x_k + \min_k x_k)/2$. Note that
$\|x-c\1\|_{\infty} = (\max_k x_k - \min_k x_k)/2=S(x)/2$. 
We then have
\[ | [Px]_i - [Px]_j |
=\Big| \sum_k (p_{ik} -p_{jk}) (x_k-c)\Big| 
\leq \|\p_i - \p_j\|_1 \cdot \|x -c\1\|_{\infty} 
\leq \epsilon S(x).\]
Since this is true for every $i$ and $j$, we obtain $S(Px)\leq \epsilon S(x)$, and the sequence $A(t-1),A(t-2),\ldots,A(0)$ results in $\epsilon$-consensus.

Conversely, suppose that the sequence of matrices $A(t-1), A(t-2), \ldots, A(0)$ results in $\epsilon$-consensus. Fix some $i$ and $j$. Let $x$ be a vector whose $k$th component is $1/2$ if $p_{ik}\geq p_{jk}$ and $-1/2$ otherwise. Note that $S(x)=1$. We have
$$\frac{1}{2} \|\p_i - \p_j\|_1
= (\p_i^T - \p_j^T)x = [Px]_i - [Px]_j \leq \epsilon S(x)=\epsilon,$$
where the last inequality made use of the $\epsilon$-consensus assumption. Thus, \ao{the} sequence of matrices $A(0), A(1), \ldots, A(t-1)$ is $\epsilon$-forgetful.
\end{proof} 

\blue{We will use Proposition \ref{forget} for the special case of Markov chains that are random walks. Given an undirected graph sequence \aot{sequence $G(0), G(1), \ldots$,
we consider the random walk on the state-space $\{1, \ldots, n\}$ which, at time $t$, jumps to a uniformly chosen random neighbor of its current state in $G(t)$. We let $A(0), A(1), \ldots$ be the associated transition probability matrices. % we have that $(i,j)$ is an edge if and only if $[A(t)]_{ij}>0$. 
We will say that a sequence of graphs is $\epsilon$-forgetful whenever \ao{the corresponding sequence of transition probability
matrices is $\epsilon$-forgetful.}}}
\jnt{Proposition \ref{forget}} allows us to reinterpret Theorem \ref{mainthm} as follows: random walks on time-varying \jnt{undirected $B$-connected}  graphs with
self-loops \red{and degree constancy} forget their initial distribution in a polynomial number of steps. 

%\jnt{Proposition \ref{forget}} also has a corollary which we will use later. It is based on the following \jnt{observation:} concatenating two sequence of graphs,
%each of which achieves $\epsilon$-consensus, results in a sequence which achieves $\epsilon^2$-consensus.  
%
%
%
%\begin{corollary} Suppose \jnt{that} a sequence of graphs is $\epsilon$-forgetful. Then, concatenating this sequence with itself $k$ times results
%in a sequence  which is $\jnt{\epsilon^k}$-forgetful. \end{corollary}

\section{A counterexample \label{counterexamples}} 
%\subsection{The unchanging degrees condition in Theorem \ref{mainthm} cannot be relaxed}  

\begin{figure}
\centering
\begin{tabular}{cc}
\epsfig{file=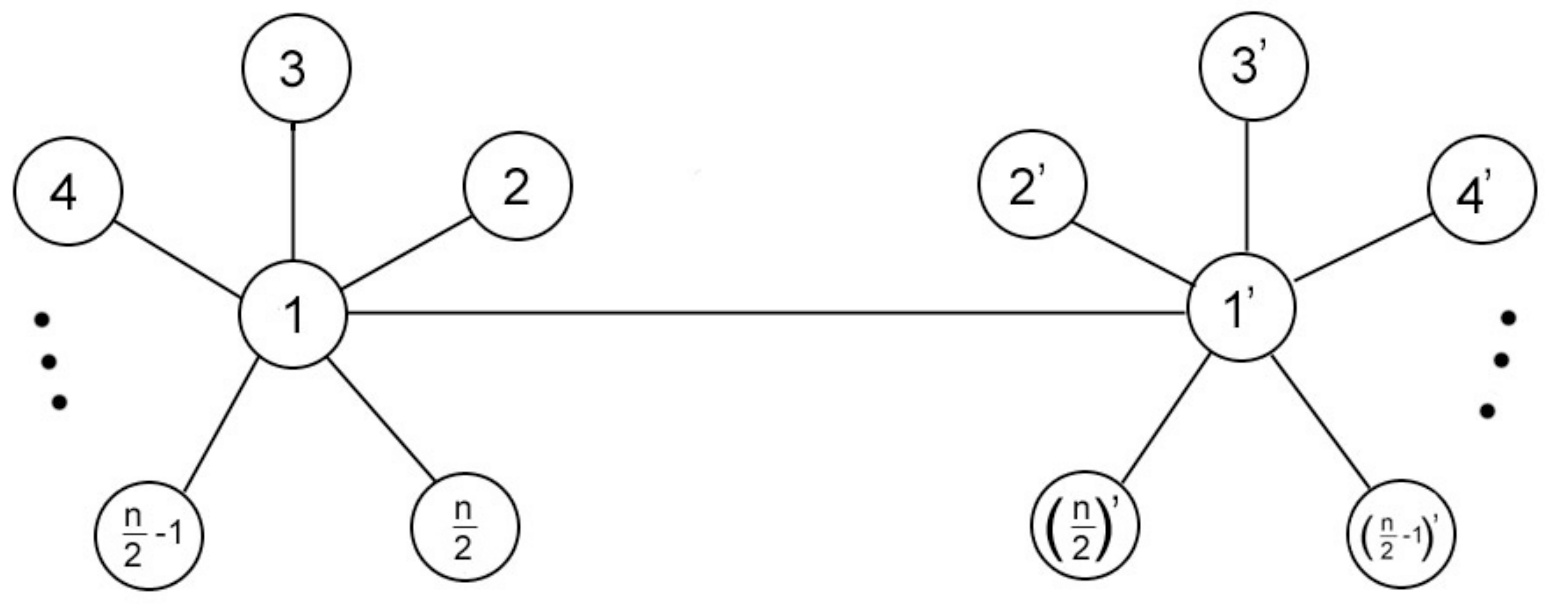,width=0.45\linewidth,clip=} &
\epsfig{file=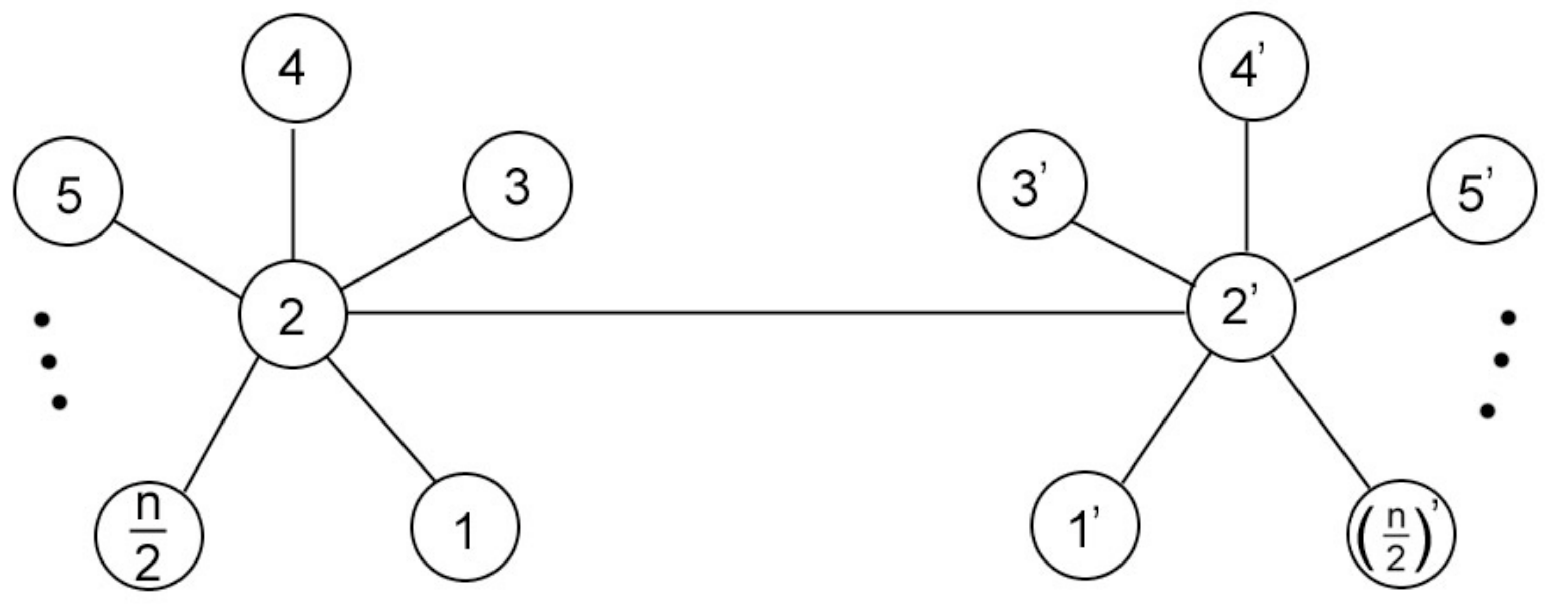,width=0.45\linewidth,clip=} \\
\epsfig{file=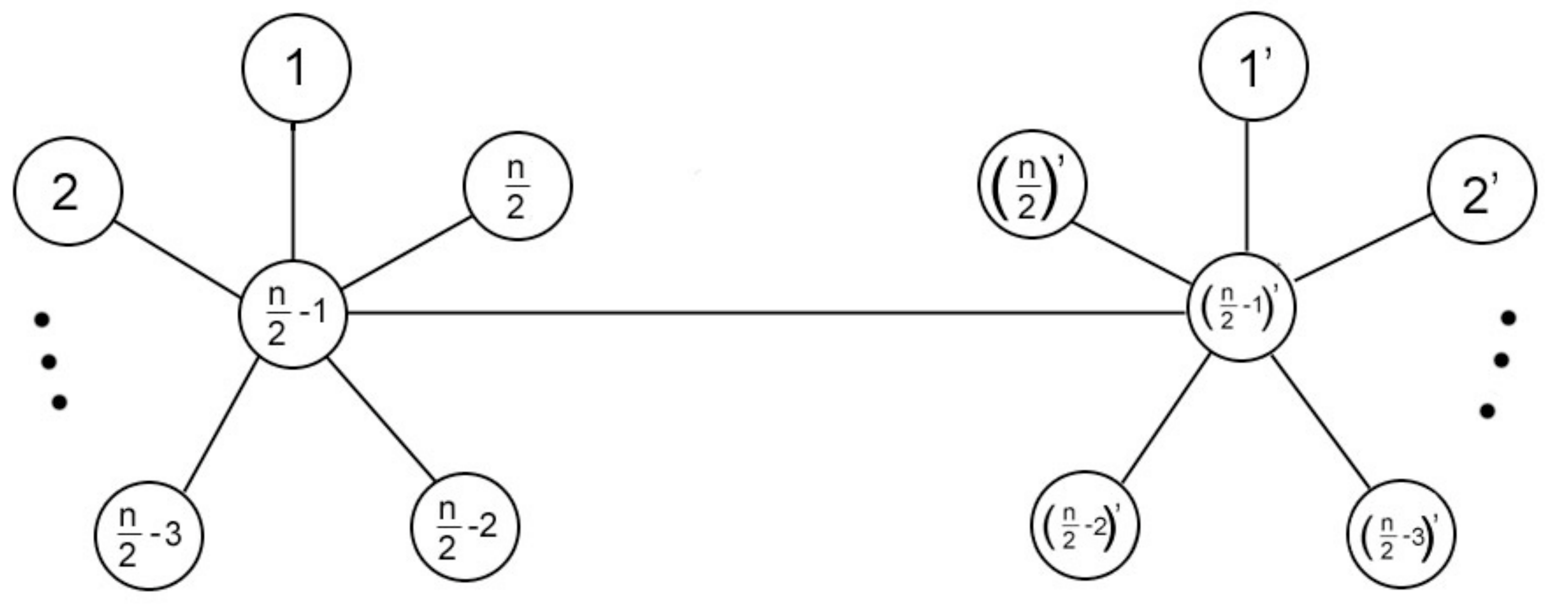,width=0.45\linewidth,clip=} &\epsfig{file=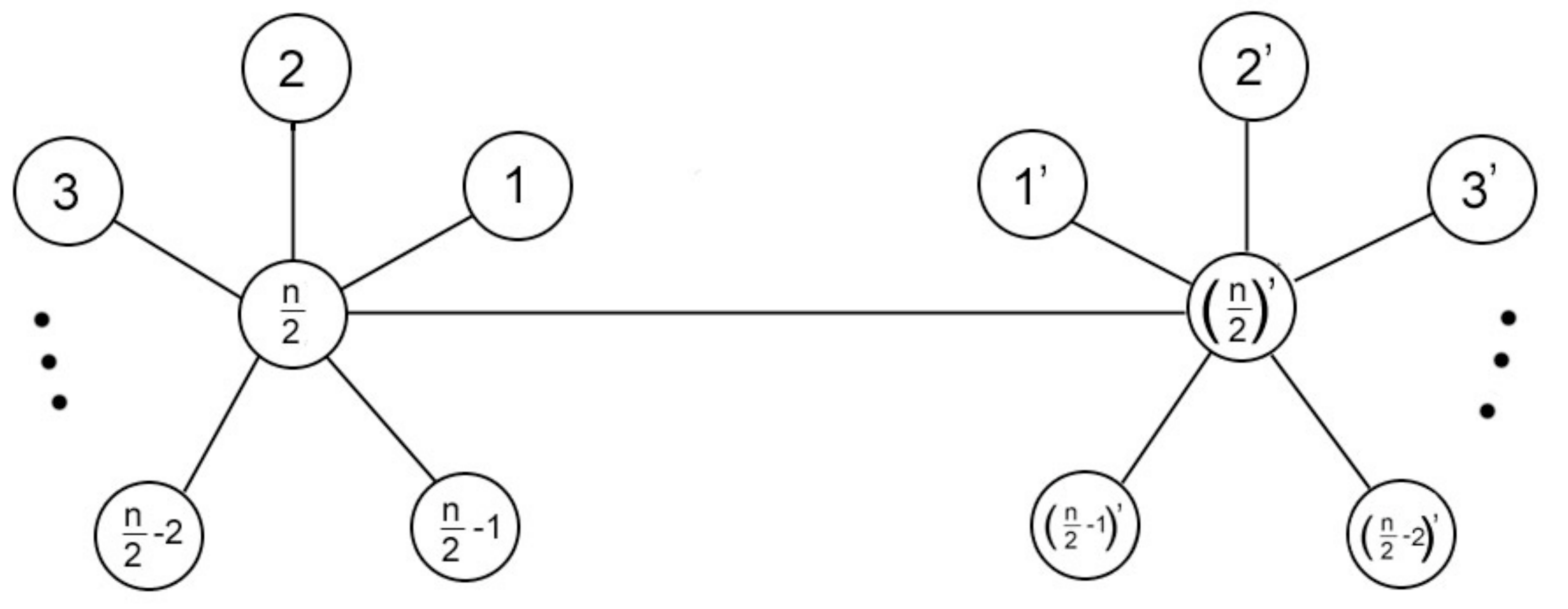,width=0.45\linewidth,clip=}
\end{tabular} \caption{The top-left figure shows graph $G(0)$; top-right
shows $G(1)$; bottom-left shows $G(\jnt{(n/2)-2})$; bottom-right shows $G(\jnt{(n/2)-1})$. As these figures illustrate, $G(t+1)$ is obtained by
applying a circular shift to each half of $G(t)$. Every node has a self-loop
which is not shown. For aesthetic reasons, instead of labeling the nodes as $1, \ldots, n$, we label them with $1, \ldots, n/2$ and $1', \ldots, (n/2)'$. }
\end{figure}

In this subsection, we show that it is impossible to relax the condition of unchanging degrees in
Theorem \ref{mainthm}. In particular, if we only impose the slightly
weaker condition that the sorted degree sequence (the non-increasing list of node degrees) does not change with time, the time to achieve $\epsilon$-consensus can grow exponentially with $n$.  This is an unpublished result of Cao, Spielman, and Morse \cite{cao}; we provide here a simple proof. We note that the graph sequence used in \jntn{the} proof (see Figure 1) is similar to the sequence used in \cite{AKL08} to prove an exponential lower bound on the cover time of time-varying graphs.

\begin{proposition} 
Let $n$ be even and let $t$  be an integer multiple of $n/2$.
Consider the graph sequence of length \blue{$t=kn/2$, consisting of periodic repetitions of \ao{the reversal\footnote{That is, we are considering the sequence $ G(n/2-1), \ldots, G(1), G(0), G(n/2-1), \ldots, G(1), G(0), G(n/2)-1, \ldots$.} of the length-$n/2$ sequence described in Figure 1}.}  For this graph sequence to result in $(1/4)$-consensus, we must have $t\geq 2^{(n/2)}/\ao{8}$. \label{degreechange}
\end{proposition}

\begin{proof} \ao{Suppose that this graph sequence of length $t$ results in $(1/4)$-consensus.} Then Proposition \ref{forget} implies that \ao{the sequence $G'$ of length $t$ consisting of periodic repetitions\footnote{That is, $G'(t)$ is the sequence $G(0), G(1) \ldots, G(n/2-1), G(0), G(1), \ldots, G(n/2-1), G(0), G(1), \ldots$.} of the length $n/2$ sequence described in Figure 1 is} $(1/4)$-forgetful. Let $p_{ij}$ be the associated $t$-step transition probabilities.

Let $T$ be the time that it takes for a random walk that starts at state $n/2$ at time $0$ to cross  
into the right-hand side  of the graph, let $\delta$ be the probability that $T$ is less than or equal to $t$, and 
define $R$ to be the set of nodes on the right side of the graph, i.e., $R=\{1' \ldots (n/2)'\}$. 
Clearly, \[ \sum_{j'\in R} p_{(n/2),j'} \leq P(T \leq t) =  \delta,\] since a walk located in $R$ at time $t$ has obviously
transitioned to the right-hand side of the graph by $t$. Next, symmetry yields  
$\sum_{j'\in R} p_{(n/2)',j'}\geq 1-\delta$. 
Using the fact that the graph sequence is $(1/4)$-forgetful in the first inequality below, we have
\[
\frac{1}{2}  \geq \sum_{j'\in R} |p_{(n/2)',j'}-p_{((n/2),j'}| \geq \sum_{j'\in R} p_{(n/2)',j'} - \sum_{j'\in R} p_{(n/2),j'} \geq (1-\delta) -\delta = 1-2\delta, 
\]
 which yields that $\delta\geq 1/4$.  By viewing periods of length $t$ as a single attempt to get to the right half of the graph, with each attempt having probability at least $1/4$ to succeed, we obtain $E[T]\leq 4t$.

So far, we have not used the structure of the graphs beyond the fact that they can partitioned into a right-side and
a left-side. We now make the observation which may be viewed as the motivation behind choosing this particular graph sequence. Let us say that node
$i$ has {\em emerged} at time $t$ if node $i$ was the center of the left-star in $G'(t-1)$; for example, node $1$ has emerged at time $1$, node $2$ has 
emerged at time $2$, and so on. By symmetry, $T$ is the expected time until a random walk starting at an emerged node crosses to the right-hand side 
of the graph. Observe that, starting from an emerged node, the random walk will transition to the right-hand side of the graph if it takes
the self-loop $n/2-1$ consecutive times and then, once it is at the center, takes the link across; however, if it fails to take the self-loop during the first
$n/2-1$ times, it then transitions to a newly emerged node. This implies that the expected time to transition to the right hand side from an emerged node  is at least the expected time until the walk takes $n/2-1$ self-loops consecutively:   $2^{(n/2)-1} \leq E[T]$. 

Putting this together with the previous inequality $E[T] \leq 4t$, we immediately have the desired result.  
\end{proof}

%
%\section{Conclusions \label{concl}}
%
%The main contribution of this paper is Theorem \ref{mainthm}, which shows that consensus algorithms converge in polynomial time \jnt{for} a large
%class of graph sequences. We also \jnt{gave simple proofs showing that} this finding is \jnt{fragile,} \ao{and even} a slight relaxation of the hypotheses cause the conclusion to fail. 
%
%Our results 
%\jnt{A similar result is available for the case of doubly stochastic update matrices $A(t)$, and our result can be viewed as complementary \cite{NOOT09}. Interestingly, both results rely on a suitable quadratic Lyapunov function as well as on the fact that all update matrices share a common left eigenvector.}
%%\comm{Removed open problem sentence.}
%%\old{However, the two proofs are otherwise quite different. It is an interesting question whether the result in this paper and the result in \cite{NOOT09} can be derived through a common proof.}

\begin{small}

\end{small}

\end{document}